\documentclass[12pt]{amsart}  

\usepackage{amsmath}
\usepackage{amsfonts}
\usepackage{amssymb}
\usepackage{latexsym}
\usepackage{graphicx}
\usepackage{subfigure}
\usepackage{color}
\usepackage{mathtools}
\usepackage{hyperref}
\usepackage{verbatim}
\usepackage[all]{xy}
\usepackage{graphics}
\usepackage{pdfsync}
\usepackage{youngtab}
\usepackage{ytableau}
\usepackage{young}
\usepackage{cancel} 
\usepackage[normalem]{ulem} 
\usepackage{tikz}
\usetikzlibrary{arrows,petri,topaths}
\usepackage{tkz-berge}
\usepackage{xcolor}
\usepackage{tikz-qtree}
\usetikzlibrary{positioning,arrows,calc,intersections,shapes,decorations}

\theoremstyle{plain}
\newtheorem{theorem}{Theorem}[section]
\newtheorem{proposition}[theorem]{Proposition}

\newtheorem{lemma}[theorem]{Lemma}
\newtheorem{corollary}[theorem]{Corollary}

\newtheorem{remark}[theorem]{Remark}

\numberwithin{equation}{section}
\newcommand{\bP}{\mathbb{Z}_+}

\newcommand{\sym}{\ensuremath{\operatorname{Sym}}}

\newcommand{\qsym}{\ensuremath{\operatorname{QSym}}}
\newcommand{\qs}{{\mathcal{S}}}		

\newcommand{\ncsa}{\mathbf{s}_{\alpha}}
\newcommand{\ncsb}{\mathbf{s}_{\beta}}
\newcommand{\ncsg}{\mathbf{s}_{\gamma}}

\newcommand{\Lc}{\mathcal{L}_{c}} 

\newcommand\partitionof[1]{\widetilde{#1}}
\newcommand\reverse[1]{{#1}^r}

\newcommand{\set}{{\sf set}} 
\newcommand{\cont}[1]{{\sf cont}(#1)} 
\newcommand{\rect}[1]{{\sf rect}(#1)} 
\newcommand{\colform}[1]{{\sf colform}(#1)} 
\newcommand{\frank}[1]{\mathrm{Frank}(#1)} 
\newcommand{\lrfrank}[1]{\mathrm{LRFrank}(#1)} 
\newcommand{\lrfranktab}[1]{\mathrm{LRFrankTab}(#1)} 
\newcommand{\rot}[1]{{\sf rotate}(#1)} 
\newcommand{\addt}{\mathfrak{t}} 

\newcommand{\suchthat}{\;|\;}

\newcommand{\cskew}{{/\!\!/}}

\newcommand{\YT}{\ensuremath{\operatorname{YT}}}
\newcommand{\LRT}{\ensuremath{\operatorname{LRT}}} 
\newcommand{\SCT}{\ensuremath{\operatorname{SCT}}}
\newcommand{\CT}{\ensuremath{\operatorname{CT}}}

\newlength\cellsize 
\savebox2{%
\begin{picture}(15,15)
\put(0,0){\line(1,0){15}}
\put(0,0){\line(0,1){15}}
\put(15,0){\line(0,1){15}}
\put(0,15){\line(1,0){15}}
\end{picture}}
\newcommand\cellify[1]{\def\thearg{#1}\def\nothing{}%
\ifx\thearg\nothing
\vrule width0pt height\cellsize depth0pt\else
\hbox to 0pt{\usebox2\hss}\fi%
\vbox to 15\unitlength{
\vss
\hbox to 15\unitlength{\hss$#1$\hss}
\vss}}
\newcommand\tableau[1]{\vtop{\let\\=\cr
\setlength\baselineskip{-16000pt}
\setlength\lineskiplimit{16000pt}
\setlength\lineskip{0pt}
\halign{&\cellify{##}\cr#1\crcr}}}
\savebox3{%
\begin{picture}(15,15)
\put(0,0){\line(1,0){15}}
\put(0,0){\line(0,1){15}}
\put(15,0){\line(0,1){15}}
\put(0,15){\line(1,0){15}}
\end{picture}}
\newcommand\expath[1]{%
\hbox to 0pt{\usebox3\hss}%
\vbox to 15\unitlength{
\vss
\hbox to 15\unitlength{\hss$#1$\hss}
\vss}}
\newcommand\bas[1]{\omit \vbox to \cellsize{ \vss \hbox to \cellsize{\hss$#1$\hss} \vss}}

\newcommand{\ptab}{{\sf P}} 
\newcommand{\qtab}{{\sf Q}} 
\newcommand{\sh}{{\sf shape}} 

\newcommand{\nsym}{\mathrm{NSym}}

\definecolor{darkbrown}{rgb}{0.0, 0.0, 1.0}
\newcommand{\bemph}[1]{\textcolor{darkbrown}{\emph{#1}}} 
\newcommand{\mbP}{\bP^{*}} 
\newcommand{\sgrp}[1]{\mathfrak{S}_{#1}} 
\newcommand{\inv}[1]{{\sf I}(#1)}
\newcommand{\crw}[1]{{\sf crw}(#1)} 
\newcommand{\stan}[1]{{\sf stan}(#1)} 
\newcommand{\pskew}{/} 
\newcommand{\cgw}[1]{{\sf cgw}(#1)} 
\newcommand{\sort}{{\sf sort}}
\newcommand{\evac}{{\sf evac}}
\begin{document}

\title[]{Noncommutative LR coefficients and crystal reflection operators}
\author{E. Richmond}
\address{Department of Mathematics, Oklahoma State University, Stillwater, OK 74708, USA}
\email{\href{mailto:edward.richmond@okstate.edu}{edward.richmond@okstate.edu}}
\author{V. Tewari}
\address{Department of Mathematics, University of Pennsylvania, Philadelphia, PA 19104, USA}
\email{\href{mailto:vvtewari@upenn.edu}{vvtewari@upenn.edu}}

\begin{abstract}
We relate noncommutative Littlewood-Richardson coefficients of Bessenrodt-Luoto-van Willigenburg to classical Littlewood-Richardson coefficients via crystal reflection operators. 
A key role is played by the combinatorics of frank words.
\end{abstract}

\maketitle

\section{Introduction}
Quasisymmetric Schur functions, introduced by Haglund-Luoto-Mason-van Willigenburg \cite{HLMvW-QS}, form a prominent basis for the Hopf algebra of quasisymmetric functions denoted by $\qsym$.
The quasisymmetric Schur function indexed by the composition $\alpha$, denoted by $\qs_{\alpha}$, is obtained by  summing monomials attached to semistandard composition tableaux of shape $\alpha$.
This is reminiscent of the definition of Schur functions as sums of monomials corresponding to semistandard Young tableaux.
As the name suggests, quasisymmetric Schur functions share many properties with classical Schur functions, and Mason's map $\rho$ defined in \cite{mason-1} connects the combinatorics of composition tableaux to that of Young tableaux.
Understanding analogues of Schur functions and their generalizations has long been a theme in algebraic combinatorics; see \cite{AllenHallamMason,Per,Assaf-Searles-Kohnert,assafsearles,BBSSZ,Monical-Pechenik-Searles, Pechenik-Searles,searles} for recent work in this context.

The Hopf algebra of noncommutative symmetric functions, $\nsym$, introduced in the seminal paper \cite{GKLLRT}, is Hopf-dual to $\qsym$ as shown by Malvenuto-Reutenauer \cite{malvenuto-reutenauer}.
Bessenrodt-Luoto-van Willigenburg \cite{BLvW} studied the dual basis elements $\ncsa$ corresponding to quasisymmetric Schur functions. The resulting functions are also indexed by compositions and are called noncommutative Schur functions. The inclusion of the Hopf algebra of symmetric functions $\sym$ into $\qsym$ induces a projection $\chi:\nsym\twoheadrightarrow \sym$. This projection maps noncommutative Schur functions to classical Schur functions, and justifies the name of the former.
The structure constants $C_{\alpha\beta}^{\gamma}$ that arise in
\begin{align}\label{eqn:noncommutative lr coefficients introduction}
\ncsa\cdot\ncsb=\sum_{\gamma}C_{\alpha\beta}^{\gamma}\, \ncsg
\end{align}
 are called noncommutative Littlewood-Richardson (LR) coefficients.
 Noncommutative LR coefficients turn out to be nonnegative integers and furthermore, they refine the classical LR coefficients $c_{\mu\nu}^{\lambda}$ that arise in the product of Schur functions
 \begin{align}\label{eqn:classical lr coefficients introduction}
 s_{\nu}\cdot s_{\mu}=\sum_{\lambda}c_{\nu\mu}^{\lambda}\, s_{\lambda}.
 \end{align}
More precisely, suppose $\nu$ and $\mu$ are partitions that rearrange to compositions $\alpha$ and $\beta$ respectively.
Applying $\chi$ to both sides in \eqref{eqn:noncommutative lr coefficients introduction} and comparing the result with  \eqref{eqn:classical lr coefficients introduction} implies 
\begin{align}\label{eqn: nc lr refines lr introduction}
	c_{\nu\mu}^{\lambda}=\sum_{\gamma}C_{\alpha\beta}^{\gamma}
\end{align}
where the sum on the right runs over all compositions $\gamma$ that rearrange to a fixed partition $\lambda$.
Among the numerous combinatorial interpretations of $c_{\nu\mu}^{\lambda}$, the one we focus on states that $c_{\nu\mu}^{\lambda}$ counts the number of LR tableaux of shape $\lambda/\mu$ and content $\nu$.
Our primary goal in this article is to understand the summands on the right hand side in
\eqref{eqn: nc lr refines lr introduction} in terms of LR tableaux. To this end, crystal reflection operators are key.




To state our result, we introduce the necessary notation briefly. The reader is referred to Section~\ref{sec:background} for details.
Given a composition $\alpha$, we denote the partition underlying $\alpha$ by $\sort(\alpha)$.
Let $\LRT(\lambda,\mu,\nu)$ be the set of LR tableaux of shape $\lambda/\mu$ and content $\nu$.
Given a permutation $\sigma$, let $\LRT^{\sigma}(\lambda,\mu,\nu)$ be the set of tableaux obtained by applying the crystal reflection operators corresponding to a reduced word of $\sigma$. Let $\alpha\coloneqq \sigma\cdot \nu$, and let $\beta$ be any composition that satisfies $\sort(\beta)=\mu$.
On applying the map $\rho_{\beta}^{-1}$  (which sends Young tableaux to composition tableaux, \cite[Chapter 4]{QSbook}) to elements in $\LRT^{\sigma}(\lambda,\mu,\nu)$, we obtain the disjoint decomposition
\begin{align}\label{eqn:skew LR disjoint}
\LRT^{\sigma}(\lambda,\mu,\nu)=\coprod_{\gamma}X_{\alpha\beta}^{\gamma},
\end{align}
where $X_{\alpha\beta}^{\gamma}$ consists of all tableaux $T\in \LRT^{\sigma}(\lambda,\mu,\nu)$ whose outer shape is given by the composition $\gamma$ under the map $\rho_{\beta}^{-1}$. Under this setup, our main theorem states the following.
\begin{theorem}\label{thm:intro_main_1}
The noncommutative LR coefficient $C_{\alpha\beta}^{\gamma}$ equals the  cardinality of $X_{\alpha\beta}^{\gamma}$.
\end{theorem}
The upshot of Theorem~\ref{thm:intro_main_1} is that starting from  $\LRT(\lambda,\mu,\nu)$, we can compute \emph{all} noncommutative LR coefficients $C_{\alpha\beta}^{\gamma}$, where $\alpha$ dictates the choice of crystal reflections to be performed and $\beta$ determines the generalized $\rho$ map to be applied.
We also interpret Theorem~\ref{thm:intro_main_1} in terms of chains in Young's lattice indexed by certain frank words and obtain a rule for $C_{\alpha\beta}^{\gamma}$ involving box-adding operators on compositions.
In the case where $\alpha$ is either a partition or reverse partition, our interpretations yield the two LR rules in \cite{BTvW}.


\medskip

\textbf{Outline of the article: } Section~\ref{sec:background} sets up all the necessary combinatorial background.
Section~\ref{sec:main_result} describes our central result along with examples. In Section~\ref{subsec:symmetric group action frank words}, we introduce frank words and describe the Lascoux-Sch\"utzenberger symmetric group action on frank words.
Section~\ref{subsec:skew shape compatibility} identifies LR tableaux  to certain distinguished frank words, called compatible words, drawing upon work by Remmel-Shimozono \cite{Remmel-Shimozono}.
Section~\ref{subsec:sym_actions_tableaux_words} relates crystal reflections acting on LR tableaux to the symmetric group action on compatible words.
Section~\ref{sec:back to CT} relates the results in Section~\ref{sec:Frank words} back to noncommutative LR coefficients. The key result in this section, which implies Theorem~\ref{thm:intro_main_1}, is Proposition~\ref{prop:crystal action permutes parts}. We conclude our article with Corollary~\ref{cor:box-adding and LR coefficients}, which reinterprets our combinatorial interpretation for noncommutative LR coefficients in terms of box-adding operators on compositions.

\section{Background}\label{sec:background}
To keep our background section brief, we assume knowledge of combinatorial structures and algorithms arising in the well-studied theory of  symmetric functions such as partitions, skew shapes, Young tableaux, Robinson-Schensted-Knuth insertion and jeu-de-taquin.
The reader is referred to the  standard texts \cite{fulton,macdonald,sagan,stanley-ec2} for further information.
Regarding notions relevant to the theory of quasisymmetric Schur functions, we adhere to the notation and conventions in \cite{QSbook}.
We emphasize here that what we call quasisymmetric Schur functions in this article are the Young quasisymmetric Schur functions in \cite{QSbook}.
We have made this choice so that the tableau objects that we consider in the quasisymmetric/noncommutative setting align with the more prevalent notion of Young tableaux rather than reverse tableaux.

\subsection{Words}
We denote the set of positive integers by $\bP$. Given $n\in \bP$, we define $[n]:=\{1,\dots,n\}$. Furthermore, we set $[0]=\emptyset$.
Let $\mbP$ denote the set of all words in the alphabet $\bP$.
Consider $w=w_1\dots w_n\in \mbP$.
We call $n$ the \bemph{length} of $w$ and denoted it by $|w|$.
The word $w_n\dots w_1$ is denoted by $\reverse{w}$.
We say that $w$ is \bemph{lattice} if every prefix of $w$ contains at least as many $i$'s as $i+1$'s  for all $i\in \bP$.
We say that $w$ is \bemph{reverse-lattice} if $\reverse{w}$ is lattice.
We call $w$ a \bemph{column word}  if $w_1>\cdots > w_n$.
An ordered pair $(i,j)$ is said to be an \bemph{inversion} of $w$ if $w_i>w_j$ and $1\leq i<j\leq n$.
We denote the set of inversions of $w$ by $\inv{w}$.
Note that there is a unique permutation $\stan{w}\in  \sgrp{n}$ such that $\inv{\stan{w}}=\inv{w}$.
We call $\stan{w}$ the \bemph{standardization} of $w$.
We denote the longest word in $\sgrp{n}$ by $w_0^{(n)}$, or simply by $w_0$ if $n$ is clear from context.
We let $\sigma\in\sgrp{n}$ act on a sequence $\lambda=(\lambda_1,\dots,\lambda_n)$ by $\sigma\cdot \lambda=(\lambda_{\sigma^{-1}(1)},\dots, \lambda_{\sigma^{-1}(n)})$.
Note that this is a right action.

\subsection{Compositions and partitions}
A finite list of nonnegative integers $\alpha = (\alpha _1, \ldots , \alpha _\ell)$ is called a \bemph{weak composition}.
If $\alpha_i>0$ for all $1\leq i\leq\ell$, then $\alpha$ is called a \bemph{composition}. If, in addition, we have $\alpha _1\geq \cdots \geq\alpha _\ell >0$, then $\alpha$ is called a \bemph{partition}.
Given $\alpha=(\alpha_1, \ldots , \alpha_\ell)$ we call the $\alpha _i$ the \bemph{parts} of $\alpha$ and the sum of the $\alpha_i$, denoted by $|\alpha|$, is called the \bemph{size} of $\alpha$.
We denote the number of parts of $\alpha$ by $\ell(\alpha$) and call it the \bemph{length} of $\alpha$.
The unique composition of length and size zero is denoted by $\varnothing$.
The partition obtained by sorting the parts of a composition $\alpha$ in weakly decreasing order is denoted by $\sort(\alpha)$.
We denote the composition $(\alpha_{\ell},\dots, \alpha_1)$ by $\reverse{\alpha}$.
The \bemph{composition diagram} of $\alpha=(\alpha_1,\ldots,\alpha_{\ell})$ is the left-justified array of boxes with $\alpha_i$ boxes in row $i$ from the bottom.
If $\alpha$ is a partition, the composition diagram of $\alpha$ coincides with the Young diagram of $\alpha$ in the French convention.
See Figure~\ref{fig:comp_diagram} for the composition diagram of $(2,1,3)$.
Recall that compositions of  $n$ are in bijection with subsets of $[n-1]$ via the map that sends $\alpha=(\alpha_1,\ldots,\alpha_{\ell})$ to  $\set(\alpha)\coloneqq \{\alpha_1,\alpha_1+\alpha_2,\ldots,\alpha_1+\cdots+\alpha_{\ell-1}\}$.
Finally, denote the \emph{transpose}  of a partition $\lambda$ by $\lambda^t$.

\begin{figure}[ht]
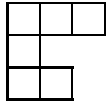

  \ytableausetup{mathmode,boxsize=1em}
  \ydiagram{3,1,2}
  \caption{The composition diagram of $(2,1,3)$.}
  \label{fig:comp_diagram}
\end{figure}

\subsection{Young tableaux} 
A  \bemph{Young tableau} $T$ of \bemph{skew shape} $\lambda/\mu$ is a filling of the boxes of
$\lambda/ \mu$ with positive integers so that the entries along the rows increase weakly read from left to right and entries along the columns increase strictly read from bottom to top.
If the entries in $T$ are all distinct and belong to $[|\lambda/\mu|]$, then we say that $T$ is \bemph{standard}.
We denote the  set of Young tableaux of shape $\lambda/\mu$ by $\YT(\lambda/\mu)$.

Given a word $w=w_1\dots w_n$, the Robinson-Schensted correspondence (via row insertion or column insertion) associates an ordered pair of Young tableaux $(\ptab(w), \qtab(w))$  of the same shape.
We call $\ptab(w)$ (respectively $\qtab(w)$) the \bemph{insertion tableau} (respectively \bemph{recording tableau}).
We {call} two words $w_1$ and $w_2$  \bemph{Knuth-equivalent} if $\ptab(w_1)=\ptab(w_2)$.
The \bemph{column reading word} of $T\in \YT(\lambda/\mu)$, denoted by $\crw{T}$, is obtained by reading the entries {of $T$} in every column from top to bottom starting from the left and going to the right.
We declare tableaux $T_1$ and $T_2$ to be \bemph{jdt-equivalent} if their column reading words are Knuth-equivalent, that is, $\ptab(\crw{T_1})=\ptab(\crw{T_2})$.

If $T$ is a Young tableau with maximal entry  $m$, then the \bemph{content} of  $T$, denoted by $\cont{T}$, is the weak composition $(\alpha_1,\dots,\alpha_m)$ where $\alpha_i$ for $1\leq i\leq m$ counts the number of times $i$ appears in $T$.
The \bemph{standardization} of $T$, denoted by $\stan{T}$, is obtained by replacing the $\alpha_i$ entries in $T$  equal to $i$ by the integers $1+\sum_{j=1}^{i-1}\alpha_j$ through $\sum_{j=1}^{i}\alpha_j$ from left to right.
Thus, $\stan{T}$ is  standard and knowing $\cont{T}$ allows us to recover $T$.
We associate a word with $\stan{T}$ by reading the entries from largest to smallest and noting the column in which they belong.
We call this word the \bemph{column growth word} of $T$ and denote it by $\cgw{T}$.
For the tableau $T$ in Figure~\ref{fig: column growth word}, we have $\cgw{T}=76564321531$.
We extend the definition of column growth word to all tableaux by setting $\cgw{T}\coloneqq \cgw{\stan{T}}$.
\begin{figure}[ht]
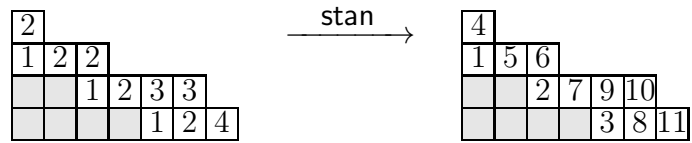

$$
\ytableausetup{mathmode, boxsize=1em}
\begin{ytableau}
2\\1& 2& 2\\*(gray!20) & *(gray!20) & 1 & 2 & 3 & 3\\ *(gray!20) & *(gray!20) & *(gray!20) & *(gray!20)  & 1 & 2 & 4
\end{ytableau}
\hspace{5mm}
\xrightarrow{\makebox[1.5cm]{{\sf stan}}}\hspace{5mm}
\begin{ytableau}
4\\1& 5& 6\\*(gray!20) & *(gray!20) & 2 & 7 & 9 & 10\\ *(gray!20) & *(gray!20) & *(gray!20) & *(gray!20) & 3 & 8 & 11
\end{ytableau}
$$
\caption{A Young tableau  and its standardization.}
\label{fig: column growth word}
\end{figure}

A \bemph{descent} of a standard Young tableau $T$ with $n$ boxes is an integer $i$ satisfying $1\leq i\leq n-1$ such that $i+1$ occupies a row strictly above that occupied by $i$.
The \bemph{descent set} of $T$ is the collection of descents of $T$, and the \bemph{descent composition} is the composition of $n$ corresponding to the descent set.
The descent set of the standard Young tableau on the right in Figure~\ref{fig: column growth word} is $\{3,8\}$.


\subsection{Composition tableaux}
To define composition tableaux, we need an analogue of Young's lattice.
The \bemph{Young composition poset} $\Lc$ is the poset on compositions where the partial order $<_c$ is obtained by taking the transitive closure of the cover relation $\lessdot_c$ defined next.
Let $\beta=(\beta_1,\ldots,\beta_m)$.
Then $ \beta \lessdot_c \alpha$ if exactly one of the following conditions holds.
\begin{itemize}
    \item $\alpha=(\beta_1,\ldots,\beta_m,1)$.
    \item $\alpha=(\beta_1,\ldots,\beta_k+1,\ldots,\beta_m)$ for some $k$ where $\beta_k\neq \beta_i$ for all $i>k$.
\end{itemize}
The reader may check that, for instance, the compositions covering $(2,1,3,2)$ in $\Lc$ are $(2,1,3,2,1)$, $(2,2,3,2)$, $(2,1,3,3)$ and $(2,1,4,2)$.
\begin{remark}\label{rem:alternate definition partial order L_c}\emph{
  The definition of  {$\lessdot_c $} implies that $\beta<_c \alpha$ if and only if for every $\beta_i\geq \beta_j$ where $i>j$ we have $\alpha_i\geq \alpha_j$.}
\end{remark}
If $\beta <_{c} \alpha$ and $\beta$ is drawn in the bottom left corner of $\alpha$, then
the \bemph{skew composition shape} $\alpha \cskew \beta$ is
defined to be the array of boxes that belong to $\alpha$ but not to $\beta$.
We refer to $\alpha$ and $\beta$ as the \bemph{outer shape} and \bemph{inner shape} respectively. If the inner shape is $\varnothing$, instead of writing $\alpha \cskew \varnothing$, we just write $\alpha$ and refer to $\alpha$ as a \bemph{straight shape}. The \bemph{size} of  $\alpha\cskew \beta$, denoted by $\vert \alpha\cskew\beta\vert$, is $\vert \alpha\vert-\vert\beta\vert$.

A \bemph{composition tableau} (abbreviated to $\CT$) $\tau$ of \bemph{shape} $\alpha \cskew \beta$ is a filling
$
\tau: \alpha\cskew \beta \longrightarrow \bP
$
that satisfies the following conditions.
\begin{enumerate}
\item The entries in each row increase weakly from left to right.
\item The entries in the leftmost column increase strictly from bottom to top.
\item For any configuration in $\tau$ of the type in Figure~\ref{fig:triple configuration}, if $a\leq c$ then $b<c$. 
\end{enumerate}
\begin{figure}[ht]
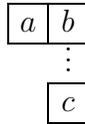

	\centering
	\begin{align*}
	\ytableausetup{mathmode,boxsize=1.25em}
	\begin{ytableau}
	a & b\\
	\none & \none[\vdots]\\
	\none & c
	\end{ytableau}
	\end{align*}
	\caption{A triple {configuration.}}
	\label{fig:triple configuration}
\end{figure}


A {composition tableau} is  \bemph{standard} if the filling $\tau$ is a bijection between $\alpha\cskew\beta$ and $[|\alpha\cskew \beta|]$.
We denote the set of CTs (respectively SCTs) of shape $\alpha\cskew \beta$ by $\CT(\alpha\cskew \beta)$ (respectively $\SCT(\alpha\cskew\beta)$).
Figure~\ref{fig:2cts} depicts a tableau in $\CT((3,6,1,7)\cskew (2,4))$ on the left, where the shaded boxes belong to the inner shape.
The \bemph{column reading word} of $\tau\in \CT(\alpha\cskew \beta)$, denoted by $\crw{\tau}$, is obtained by reading the entries in every column in {decreasing} order, starting from the leftmost column and going to the right.

Given a composition $\alpha = (\alpha _1, \dots , \alpha _{k})$, the \bemph{canonical  composition tableau} $\tau _\alpha$ is constructed by filling the boxes in the $i$-th row of the composition diagram of $\alpha$ with consecutive positive integers from $1+\sum_{j=1}^{i-1}\alpha_j$ to $\sum_{j=1}^{i}\alpha_j$ from left to right,  for $1\leq i\leq k$.
Figure~\ref{fig:2cts} shows the canonical CT of shape $(4,2,3,1)$ on the right.
\begin{figure}[!htbp]
	\centering
	\begin{align*}
	\ytableausetup{mathmode,boxsize=1em}
  \begin{ytableau}
	2 & 2& 2& 2& 3& 3 & 4\\
	1\\
	*(gray!20) & *(gray!20) & *(gray!20) & *(gray!20) & 1 & 2\\
	*(gray!20) &*(gray!20) & 1
\end{ytableau}\hspace{20mm}
	\begin{ytableau}
	10\\
	7 & 8 & 9\\
	5 & 6\\
	1 & 2 & 3 & 4
	\end{ytableau}
	\end{align*}
	\caption{An SCT of shape $( 3,6,1,7)\cskew (2,4)$ (left) and the canonical CT of shape $(4,2,3,1)$ (right).}
	\label{fig:2cts}
\end{figure}


\subsection{The \texorpdfstring{$\rho$}{rho} map}
Next we discuss a crucial map that establishes the bridge between the combinatorics of composition tableaux and that of Young tableaux.
Let $\CT(-\cskew \beta)$ denote the set of all $\CT$s with inner shape $\beta$ and $\YT(-/\partitionof{\beta})$ denote the set of all $\YT$s with inner shape $\sort(\beta)$.
Then the map $\rho _{\beta}: \CT(-\cskew \beta) \to \YT(- / \sort(\beta))$, which generalizes the map for semistandard skyline fillings \cite{mason-1} and is introduced in \cite[Chapter 4]{QSbook}, is defined as follows.
Given $\tau\in \CT(-\cskew \beta)$, obtain $\rho_\beta(\tau)$ by writing the entries in each column in increasing order from bottom to top and bottom-justifying these new columns on the inner shape $\sort(\beta)$, which might be empty.

The inverse map $
\rho ^{-1}_\beta : \YT(-/\sort({\beta}))\to \CT(-\cskew \beta)
$ is also straightforward to define.
Given $T\in \YT(-/\sort({\beta}))$,
\begin{enumerate}
\item take the set of $i$ entries in the leftmost column of $T$ and write them in increasing order in rows $\ell(\beta)+1, 2, \ldots, \ell(\beta)+i$ above the inner shape $\beta$ in the first column  to form the leftmost column of $\tau$,
\item take the set of entries in column 2 in increasing order and place them in the row with the largest index so that either
\begin{itemize}
\item the box to the immediate left of the number being placed is filled and the row entries weakly increase when read from left to right, or
\item the box to the immediate left of the number being placed belongs to the inner shape,
\end{itemize}
\item repeat the previous step with the set of entries in column $k$ for $k= 3, \ldots , m$ where $m$ is the largest part of $\sort(\beta)$.
\end{enumerate}
In the case $\beta=\varnothing$, the map $\rho_{\beta}$ is  Mason's shift map \cite{mason-1} (also known as the $\rho$ map), and we set $\rho\coloneqq\rho_{\varnothing}$.
The reader may verify that the Young tableau $T$ on the left in Figure~\ref{fig: column growth word} maps to the composition tableau on the left in Figure~\ref{fig:2cts} under $\rho_{(2,4)}^{-1}$.

\begin{remark}\emph{
  In view of the map $\rho_{\beta}$, all combinatorial notions discussed in the context of Young tableaux are inherited by composition tableaux, and we refrain from discussing all except one.
 Given $\tau\in \CT(\alpha\cskew\beta)$, its \bemph{rectification}, denoted by $\rect{\tau}$, is defined to be $\rho^{-1}(\rect{\rho_{\beta}(\tau)})$.
 We make note of one important consequence: Let $\mu$ be a partition, and $\beta^{(1)}$, $\beta^{(2)}$ be compositions such that $\sort(\beta^{(i)})=\mu$ for $i=1,2$.
 Given $T\in \YT(\lambda/\mu)$, let $\tau_{i}=\rho_{\beta^{(i)}}^{-1}(T)$ for $i=1,2$.
 As $\crw{\tau_1}=\crw{\tau_2}$, we infer that $\rect{\tau_1}=\rect{\tau_2}$.}
\end{remark}

\subsection{Classical and noncommutative LR coefficients}\label{subsec:LR coeffs}
We refer the reader to \cite{GKLLRT,QSbook,Gessel} for background on noncommutative symmetric functions and quasisymmetric functions.
Recall that the classical LR rule provides a combinatorial way to compute the structure coefficients $c_{\nu\mu}^{\lambda}$ in
\[
s_{\nu}s_{\mu}=\sum_{\lambda\vdash |\mu|+|\nu|}c_{\nu\mu}^{\lambda}s_{\lambda}.
\]
These same structure coefficients also arise in the expansion of the skew Schur function $s_{\lambda/\mu}$.
The LR rule was stated by Littlewood-Richardson \cite{Littlewood-Richardson} in 1934, and the first rigorous proofs were obtained by Thomas \cite{thomas} and Sch\"utzenberger \cite{schutzenberger} four decades later.
We call $T\in \YT(\lambda/\mu)$ a \bemph{Littlewood-Richardson tableau} (henceforth \bemph{LR tableau}) if $\cont{T}$ is a partition and $\crw{T}$ is reverse lattice.
The set of LR tableaux of shape $\lambda\pskew\mu$ and content $\nu$ is denoted by
$\LRT(\lambda,\mu,\nu)$.
The classical LR coefficient $c_{\nu\mu}^{\lambda}$ equals $|\LRT(\lambda,\mu,\nu)|$.
Figure~\ref{fig:LR-examples} depicts the LR tableaux that contribute to $c_{(4,3,1)(6,4,4)}^{(7,6,4,3,2)}$.
For various interesting combinatorial interpretations of LR coefficients, the reader is referred to \cite{Leeuwen,fomin-greene-LR}. For a Hopf-algebraic perspective, see \cite{LamLauveSottile}, and for a beautiful unifying polytopal perspective,  see \cite{Pak-Vallejo}.

\begin{figure}[ht]
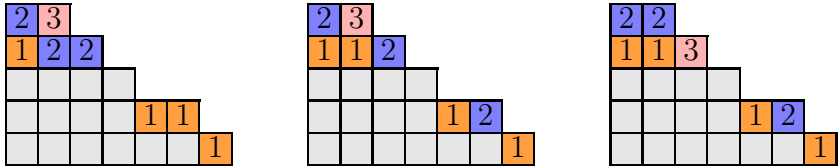

\begin{align*}
\ytableausetup{mathmode,boxsize=1em}
\begin{ytableau}
*(blue!50)2 & *(red!30)3\\
*(orange!75)1 & *(blue!50)2 & *(blue!50)2\\
*(gray!20) &*(gray!20) & *(gray!20) &*(gray!20)\\
*(gray!20) &*(gray!20) &*(gray!20) &*(gray!20) & *(orange!75)1 & *(orange!75)1\\
*(gray!20) &*(gray!20) &*(gray!20) &*(gray!20) &*(gray!20) &*(gray!20) & *(orange!75)1
\end{ytableau}
\hspace{10mm}
\begin{ytableau}
*(blue!50)2 & *(red!30)3\\
*(orange!75)1 & *(orange!75)1 & *(blue!50)2\\
*(gray!20) &*(gray!20) & *(gray!20) &*(gray!20)\\
*(gray!20) &*(gray!20) &*(gray!20) &*(gray!20) & *(orange!75)1 & *(blue!50)2\\
*(gray!20) &*(gray!20) &*(gray!20) &*(gray!20) &*(gray!20) &*(gray!20) & *(orange!75)1\\
\end{ytableau}
\hspace{10mm}
\begin{ytableau}
*(blue!50)2 & *(blue!50)2\\
*(orange!75)1 & *(orange!75)1 & *(red!30)3\\
*(gray!20) &*(gray!20) & *(gray!20) &*(gray!20)\\
*(gray!20) &*(gray!20) &*(gray!20) &*(gray!20) & *(orange!75)1 & *(blue!50)2\\
*(gray!20) &*(gray!20) &*(gray!20) &*(gray!20) &*(gray!20) &*(gray!20) & *(orange!75)1\\
\end{ytableau}
\end{align*}
\caption{The three LR tableaux contributing to $c_{(4,3,1)(6,4,4)}^{(7,6,4,3,2)}$.}
\label{fig:LR-examples}
\end{figure}

To describe the noncommutative LR rule, we need noncommutative analogues of Schur functions or, equivalently,
quasisymmetric analogues of skew Schur functions.
Following \cite[Proposition 5.2.6]{QSbook}, we define the \bemph{skew quasisymmetric Schur function} indexed by  $\alpha\cskew \beta$ to be
\begin{align}\label{eqn:def_skew_qs}
\qs_{\alpha\cskew\beta}\coloneqq\sum_{\tau\in \CT(\alpha\cskew\beta)}\mathbf{x}^{\cont{\tau}}.
\end{align}
In \eqref{eqn:def_skew_qs}, $\cont{\tau}$ refers to the content of $\tau$ and $\mathbf{x}^{\cont{\tau}}\coloneqq x_1^{\alpha_1}\cdots x_m^{\alpha_m}$ where $m$ is the largest entry in $\tau$ and $(\alpha_1,\dots,\alpha_{m})=\cont{\tau}$.
If $\beta=\varnothing$ in \eqref{eqn:def_skew_qs}, instead of writing $\qs_{\alpha\cskew \varnothing}$, we write $\qs_{\alpha}$ and call this the \bemph{quasisymmetric Schur function} indexed by $\alpha$.
Additionally, we set $\qs_\varnothing =1$.


The noncommutative Schur functions are defined indirectly \cite[Definition 5.6.1]{QSbook} as elements of the basis in $\nsym$ dual  to the basis of quasisymmetric Schur functions in $\qsym$.
We state  the LR rule for noncommutative Schur functions, equivalent to \cite[Theorem 5.6.2]{QSbook}.
\begin{theorem}\label{thm:nclr_rule_blvw}
  Let $\alpha,\beta$ be compositions. Then
  \[
  \ncsa\ncsb=\sum_{\gamma\vDash |\alpha|+|\beta|} C_{\alpha\beta}^{\gamma}\, \ncsg,
  \]
  where $C_{\alpha\beta}^{\gamma}$ is the number of $\SCT$s of shape $\gamma\cskew\beta$ that rectify to  $\tau_{\alpha}$.
\end{theorem}
As alluded to in the introduction, one of our  primary motivations for this article is a description for noncommutative LR coefficients involving LR tableaux.
In particular, we seek an appropriate analogue to LR tableaux in the context of composition tableaux.

\subsection{Crystal reflection operators and LR tableaux}
For an in-depth exposition on crystal bases and their relevance in algebraic combinatorics and representation theory, we refer the reader to \cite{bump-schilling}. We proceed to describe crystal reflection operators as they are pertinent for our purposes.

Given a positive integer $i$, we define the \bemph{crystal reflection operator}  $s_i$  acting  on the set of Young tableaux as follows.
Let $T\in \YT(\lambda/\mu)$, and let $w=w_1\dots w_n\coloneqq\crw{T}$.
Scan $w$ from left to right and pair each $i+1$ with the closest unpaired $i$ that follows.
If no further pairing is possible, change all  unpaired $i$'s to $i+1$'s or vice versa depending on whether the number of $i$'s is greater than the number of $i+1$'s or not.
 Say the new word obtained via this procedure is $w'$.
Define $s_i(T)$ to be the unique YT of shape $\lambda\pskew \mu$ such that $\crw{s_i(T)}=w'$.
 Lascoux-Sch\"utzenberger \cite{Lascoux-Schutzenberger-1} (see also \cite[Proposition 9]{Reiner-Shimozono-plactification} and \cite[Section 3]{Lascoux-Schutzenberger-keys}) proved that the operators $s_i$ define an action of the (infinite) symmetric group on $\YT(\lambda/\mu)$  by establishing the following relations.
\begin{align}
  s_i^2&=Id,\nonumber\\
  s_is_j&=s_js_i \text{ if } |i-j|>1,\nonumber\\
  s_is_{i+1}s_i&=s_{i+1}s_is_{i+1}.
\end{align}
These relations imply that $\sigma(T)$ is well defined for any permutation $\sigma$. In partitcular, to compute $\sigma(T)$,  let $s_{i_1}\cdots s_{i_k}$ be any reduced word for $\sigma$ and compute $s_{i_1}\cdots s_{i_k}(T)$.

For  $\sigma\in \sgrp{\ell(\nu)}$, define
\begin{align}\label{eqn:lr_tableaux_crystal_action}
\LRT^{\sigma}(\lambda,\mu,\nu)\coloneqq\{\sigma(T)\suchthat T\in \LRT(\lambda,\mu,\nu)\}.
\end{align}
Since crystal reflection operators define an $\sgrp{|\ell(\nu)|}$-action, we have
\begin{align}\label{eqn:lr_coefficient_crystal}
|\LRT^{\sigma}(\lambda,\mu,\nu)|=|\LRT(\lambda,\mu,\nu)|=c_{\nu\mu}^{\lambda},
\end{align}
for all permutations $\sigma \in \sgrp{\ell(\nu)}$.
Figure~\ref{fig:crystal action on LR tableaux} shows all tableaux in $\LRT^{\sigma}(\lambda,\mu,\nu)$ where $\lambda=(7,6,4,3,2)$, $\mu=(6,4,4)$, $\nu=(4,3,1)$ and $\sigma=s_1s_2$.
The tableaux in $\LRT(\lambda,\mu,\nu)$ are shown in Figure~\ref{fig:LR-examples}.
Note that all tableaux in Figure~\ref{fig:crystal action on LR tableaux} have content $(1,4,3)=\sigma\cdot (4,3,1)$.

\begin{figure}[ht]
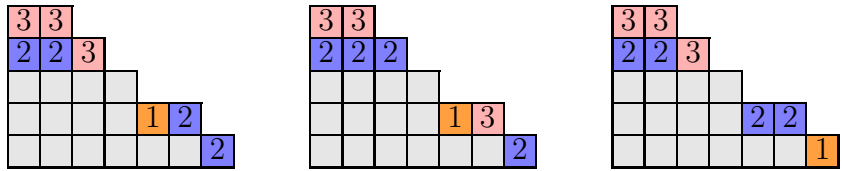

\begin{align*}
  \ytableausetup{mathmode,boxsize=1em}
\begin{ytableau}
*(red!30)3 & *(red!30)3\\
*(blue!50)2 & *(blue!50)2 & *(red!30)3\\
*(gray!20) &*(gray!20) & *(gray!20) &*(gray!20)\\
*(gray!20) &*(gray!20) &*(gray!20) &*(gray!20) & *(orange!75)1 & *(blue!50)2\\
*(gray!20) &*(gray!20) &*(gray!20) &*(gray!20) &*(gray!20) &*(gray!20) & *(blue!50)2\\
\end{ytableau}
\hspace{10mm}
\begin{ytableau}
*(red!30)3 & *(red!30)3\\
*(blue!50)2 & *(blue!50)2 & *(blue!50)2\\
*(gray!20) &*(gray!20) & *(gray!20) &*(gray!20)\\
*(gray!20) &*(gray!20) &*(gray!20) &*(gray!20) & *(orange!75)1 & *(red!30)3\\
*(gray!20) &*(gray!20) &*(gray!20) &*(gray!20) &*(gray!20) &*(gray!20) & *(blue!50)2\\
\end{ytableau}
\hspace{10mm}
\begin{ytableau}
*(red!30)3 & *(red!30)3\\
*(blue!50)2 & *(blue!50)2 & *(red!30)3\\
*(gray!20) &*(gray!20) & *(gray!20) &*(gray!20)\\
*(gray!20) &*(gray!20) &*(gray!20) &*(gray!20) & *(blue!50) 2 & *(blue!50)2\\
*(gray!20) &*(gray!20) &*(gray!20) &*(gray!20) &*(gray!20) &*(gray!20) & *(orange!75)1\\
\end{ytableau}
\end{align*}
\caption{Crystal reflection corresponding to $s_1s_2$ on the tableaux in Figure~\ref{fig:LR-examples}.}
\label{fig:crystal action on LR tableaux}
\end{figure}

It transpires that even though the elements in $\LRT^{\sigma}(\lambda,\mu,\nu)$ for an arbitrary permutation $\sigma$ shed no further light on classical LR coefficients, they do carry crucial information as far as computing noncommutative LR coefficients is concerned.
To motivate our upcoming results and to establish a connection with Theorem~\ref{thm:nclr_rule_blvw}, we invite the reader to check that  for all tableaux $T$ in Figure~\ref{fig:crystal action on LR tableaux}, we have that $\rho^{-1}(\rect{T}) $ is the CT in Figure~\ref{fig:rectification_lr_tableaux}.
Note that this CT is the unique CT with shape and content equaling  $(1,4,3)$.
Equivalently, we could say that  $\rho^{-1}(\rect{\stan{T}}) $ is the canonical CT of shape $(1,4,3)$.

\begin{figure}[ht]
\begin{align*}
  \ytableausetup{mathmode,boxsize=1em}
\begin{ytableau}
3 & 3& 3\\
2 & 2 &2 & 2\\
1
\end{ytableau}
\end{align*}
\caption{The rectification of all tableaux in Figure~\ref{fig:crystal action on LR tableaux}.}
\label{fig:rectification_lr_tableaux}
\end{figure}

In Section~\ref{sec:Frank words}, we establish this property in general by employing various properties of frank words and their relation to jeu-de-taquin. Prior to that, we describe our main result and illustrate it with examples.
\section{Main result}\label{sec:main_result}

Our first main result provides a combinatorial description for $C_{\alpha\beta}^{\gamma}$ using crystal reflection operators.
We state here the main theorem and provide a proof which assumes results we establish in Section~\ref{sec:Frank words}.
\begin{theorem}\label{thm:main result 1}
Let $\alpha,\,\beta$ and $\gamma$ be compositions such that $|\beta|+|\alpha|=|\gamma|$. Let $(\lambda,\mu,\nu)$ be the triple $(\sort(\gamma), \sort(\beta),\sort(\alpha))$.
If $\sigma\in \mathfrak{S}_{\ell(\nu)}$ is a permutation such that $\alpha=\sigma\cdot \nu$, then
\[
C_{\alpha\beta}^{\gamma}=|\{T\in \LRT^{\sigma}(\lambda,\mu,\nu) \suchthat \sh(\rho_{\beta}^{-1}(T))=\gamma\cskew \beta\}|
\]
\end{theorem}
\begin{proof}
Let $(\lambda,\mu,\nu)$ be as in the statement of the theorem.
Recall from \eqref{eqn: nc lr refines lr introduction} in the introduction that
\begin{align}\label{eqn:nc lr refines lr}
c_{\nu\mu}^{\lambda}=\sum_{\sort(\delta)=\lambda}C_{\alpha\beta}^{\delta},
\end{align}
where we work with the convention that $C_{\alpha\beta}^{\delta}=0$ if $\delta$ does not lie above $\beta$ in $\Lc$.
Given $\gamma$ such that $\sort(\gamma)=\lambda$, define
\begin{align}
X_{\gamma}\coloneqq\{T\in \LRT^{\sigma}(\lambda,\mu,\nu)\suchthat \sh(\rho_{\beta}^{-1}(\stan{T}))=\gamma\}.
\end{align}
By Proposition~\ref{prop:crystal action permutes parts}, we have that $|X_{\gamma}|\leq C_{\alpha\beta}^{\gamma}$.
On the other hand, from \eqref{eqn:nc lr refines lr} and the fact that  $\sh(\rho_{\beta}^{-1}(\stan{T}))= \sh(\rho_{\beta}^{-1}(T))$ for all Young tableaux $T$ with inner shape $\mu$, we infer that
\begin{align}
|\LRT^{\sigma}(\lambda,\mu,\nu)|=\sum_{\sort(\delta)=\lambda}|X_{\delta}|.
\end{align}
Thus, we must have $|X_{\gamma}|=C_{\alpha\beta}^{\gamma}$.
\end{proof}

Theorem~\ref{thm:main result 1} gives us a way to compute $C_{\alpha\beta}^{\gamma}$ using crystal reflections and the generalized $\rho$ map provided we know $\LRT(\lambda,\mu,\nu)$.
We discuss an example next to illustrate Theorem~\ref{thm:main result 1}.
In \eqref{eqn:main_demo_1}, we depict $\rho_{(1,2)}^{-1}\circ s_1(T)$ for the $T$ in $\LRT((5,3,2),(2,1),(4,2,1))$. 
By considering the shapes of resulting CTs, we infer that $C^{(3,5,2)}_{(2,4,1)(1,2)}=1$ and $C^{(2,5,3)}_{(2,4,1)(1,2)}=1$.

\begin{align}\label{eqn:main_demo_1}
    \ytableausetup{mathmode,boxsize=1em}
      \begin{ytableau}
        2 &3\\ *(gray!20) & 1 &2\\ *(gray!20) & *(gray!20) & 1 & 1 & 1
      \end{ytableau}
     \xrightarrow{\makebox[2cm]{$s_1$}}
     \begin{ytableau}
       2 &3\\ *(gray!20) & 1 &2\\ *(gray!20) & *(gray!20) & 1 & 2 & 2
     \end{ytableau}
     \xrightarrow{\makebox[2cm]{$\rho_{(1,2)}^{-1}$}}
     \begin{ytableau}
       2 &3\\ *(gray!20) & *(gray!20) &1 & 2& 2\\ *(gray!20) & 1 & 2
     \end{ytableau}
     \nonumber\\
     \begin{ytableau}
       1 &3\\ *(gray!20) & 2 &2\\ *(gray!20) & *(gray!20) & 1 & 1 & 1
     \end{ytableau}
     \xrightarrow{\makebox[2cm]{$s_1$}}
     \begin{ytableau}
       2 &3\\ *(gray!20) & 2 &2\\ *(gray!20) & *(gray!20) & 1 & 1 & 2
     \end{ytableau}
     \xrightarrow{\makebox[2cm]{$\rho_{(1,2)}^{-1}$}}
     \begin{ytableau}
       2 &2 & 2\\ *(gray!20) & *(gray!20) &1 & 1& 2\\ *(gray!20)  & 3
     \end{ytableau}
\end{align}
On the other hand, if we compute $\rho_{(2,1)}^{-1}\circ s_1(T)$ for the same LR tableaux, then we obtain the two CTs in \eqref{eqn:main_demo_2}. We conclude that $C^{(3,5,2)}_{(2,4,1)(2,1)}=1$  and $C^{(5,2,3)}_{(2,4,1)(2,1)}=1$.

\begin{align}\label{eqn:main_demo_2}
    \ytableausetup{mathmode,boxsize=1em}
      \begin{ytableau}
        2 &3\\ *(gray!20) & 1 &2\\ *(gray!20) & *(gray!20) & 1 & 1 & 1
      \end{ytableau}
     \xrightarrow{\makebox[2cm]{$s_1$}}
     \begin{ytableau}
       2 &3\\ *(gray!20) & 1 &2\\ *(gray!20) & *(gray!20) & 1 & 2 & 2
     \end{ytableau}
     \xrightarrow{\makebox[2cm]{$\rho_{(2,1)}^{-1}$}}
     \begin{ytableau}
       2 &3\\ *(gray!20)  &1 & 1& 2& 2\\ *(gray!20) & *(gray!20) & 2
     \end{ytableau}
     \nonumber\\
     \begin{ytableau}
       1 &3\\ *(gray!20) & 2 &2\\ *(gray!20) & *(gray!20) & 1 & 1 & 1
     \end{ytableau}
     \xrightarrow{\makebox[2cm]{$s_1$}}
     \begin{ytableau}
       2 &3\\ *(gray!20) & 2 &2\\ *(gray!20) & *(gray!20) & 1 & 1 & 2
     \end{ytableau}
     \xrightarrow{\makebox[2cm]{$\rho_{(2,1)}^{-1}$}}
     \begin{ytableau}
       2 &2 &2\\ *(gray!20)  &3\\ *(gray!20) & *(gray!20)  & 1 & 1 & 2
     \end{ytableau}
\end{align}

\medskip

\section{Frank words and LR tableaux}\label{sec:Frank words}
In order to prove Theorem~\ref{thm:main result 1}, we need to understand the rectification of tableaux in $\LRT^{\sigma}(\lambda,\mu,\nu)$ followed by an application of the generalized $\rho$ map.
To this end, we study various aspects of column growth words of these tableaux.
In particular, we identify these words as certain frank words that satisfy an additional compatibility condition. This characterization eventually allows us to connect $\LRT^{\sigma}(\lambda,\mu,\nu)$ to the computation of noncommutative LR coefficients.

\subsection{A symmetric group action on frank words}\label{subsec:symmetric group action frank words}
Frank words were introduced by Lascoux-Sch\"utzenberger \cite{Lascoux-Schutzenberger-keys} in their investigation of key polynomials.
Subsequently, Reiner and Shimozono \cite{Reiner-Shimozono-Keys} studied the combinatorics of frank words in depth in the context of a flagged Littlewood-Richardson rule, and we follow their exposition as far as notions in this section are concerned.

Given a nonempty word $w\in \mbP$, consider its factorization $w^{(1)}w^{(2)}\cdots w^{(m)}$ where each $w^{(i)}$ is a maximal column word (necessarily nonempty).
We call $w$ an \bemph{$m$-column word}.
Define the \bemph{column form} of $w$ to be the composition $\colform{w}\coloneqq(|w^{(1)}|,\dots,|w^{(m)}|) $.
We say that $w$ is \bemph{frank} if $\ptab(w)$ is of shape $\lambda^{t}$ where $\lambda=\sort(\colform{w})$.
Given a composition $\alpha$, denote the set of frank words $w$ satisfying $\colform{w}=\alpha$ by $\frank{\alpha}$.
For instance, $w= \framebox{432}\hspace{1mm}\framebox{32} \hspace{1mm} \framebox{6531}$ is a $3$-column word with $\colform{w}=(3,2,4)$. Here, and henceforth, we will put frames around maximal columns words.
Figure~\ref{fig:w is frank} depicts $\ptab(w)$.
Note that the shape underlying it is $(4,3,2)^t$.
Therefore, $w$ is frank and belongs to $\frank{(3,2,4)}$.
We proceed to describe a symmetric group action on frank words that we later connect to the symmetric group action on Young tableaux described earlier.
This action is best understood by focusing on 2-column frank words.

\begin{figure}[!htbp]
$$
\ytableausetup{mathmode,boxsize=1em}
\begin{ytableau}
4\\3 & 6\\2 & 3 & 5\\1 & 2 & 3
\end{ytableau}
$$
\caption{The insertion tableau corresponding to $w=\framebox{432}\hspace{1mm} \framebox{32} \hspace{1mm} \framebox{6531}$.}
\label{fig:w is frank}
\end{figure}

Let $A$ denote the set of 2-column frank words.
Consider $w\in A$ and let $\colform{w}=(\beta_1,\beta_2)$.
By \cite[Appendix 2]{Reiner-Shimozono-Keys}, we have that $w$ may be identified as the column reading word of a tableau $T$ of shape $(\beta_1,\beta_2)^t$ if  $\beta_1\geq \beta_2 $ or of a tableau $T$ of shape $(\beta_2,\beta_2)^t/(\beta_2-\beta_1)^t$ if $\beta_1<\beta_2$.
If the former holds, define $\iota(w)$ to be the column reading word of the unique tableau $ T'$  of shape $(\beta_1,\beta_1)^t/(\beta_1-\beta_2)^t$ that is jdt-equivalent to $T$
(obtained by performing jeu-de-taquin slides within the rectangle $(\beta_1,\beta_1)^t$).
If the latter holds, define $\iota(w)$ to be the column reading word of the unique tableau $ T'$  of shape $(\beta_2,\beta_1)^t$ that is jdt-equivalent to $T$.
Clearly, $\iota$  is an involution on $A$. Equally importantly, $w$ and $\iota(w)$ are Knuth-equivalent.
For instance, jdt-equivalence of the tableaux in Figure~\ref{fig:jdt on two columns} implies that $\iota(\framebox{76421} \hspace{1mm} \framebox{632})=\framebox{621}\hspace{1mm}\framebox{76432}$.
\begin{figure}[ht]
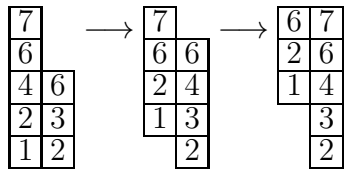

$$
\ytableausetup{mathmode,boxsize=1em}
\begin{ytableau}
7\\6\\4 & 6\\2 & 3\\1&2
\end{ytableau}
\longrightarrow
\begin{ytableau}
7\\6&6\\2& 4\\1 & 3\\\none&2
\end{ytableau}
\longrightarrow
\begin{ytableau}
6 & 7\\2 & 6\\1 & 4\\\none & 3\\\none & 2
\end{ytableau}
$$
\caption{Establishing $\iota(\framebox{76421}\hspace{1mm}\framebox{632})=\framebox{621} \hspace{1mm} \framebox{76432}$ by jeu-de-taquin.}
\label{fig:jdt on two columns}
\end{figure}

We employ the involution $\iota$ to construct the desired symmetric group action.
Let $\lambda$ be a partition and  let $m\coloneqq\ell(\lambda)$. 
Define
\[
\mathcal{F}_{\lambda}\coloneqq \coprod_{\sort(\beta)=\lambda}\frank{\beta}.
\]
Following \cite{Lascoux-Schutzenberger-keys}, define an action of $\sgrp{m}$ on $\mathcal{F}_{\lambda}$ by describing the action of the generator $s_i$ for $1\leq i\leq m-1$ as follows:
Let $w^{(1)}\cdots w^{(m)}$ be the maximal column word factorization of $w\in \mathcal{F}_{\lambda}$.
Let $v^{(i)}v^{(i+1)}\coloneqq\iota(w^{(i)}w^{(i+1)})$ and define $v=w^{(1)}\cdots w^{(i-1)}v^{(i)}v^{(i+1)}w^{(i+2)}\cdots w^{(m)}$.
Observe that $\sort(\colform{v})=\lambda$. As $v$ and $w$ are Knuth-equivalent, we infer that $v\in \mathcal{F}_{\lambda}$ .
We define $v$ to be $s_i(w)$.
This given, we may now define $\sigma(w)$ for any $\sigma\in \mathfrak{S}_m$ by making a choice of reduced word for $\sigma$.

\subsection{\texorpdfstring{$\lambda/\mu$}{lskewm}-compatible frank words and LR tableaux}\label{subsec:skew shape compatibility}
Given partitions $\lambda,\mu$ such that $\mu\subseteq \lambda$ and a composition $\alpha$, we say that $w\in \frank{\alpha}$ is \bemph{$\lambda/\mu$-compatible} if for every suffix $w'$ of $w$, we have that $\cont{w'}+\mu^t$ is a partition and that $\cont{w}+\mu^t=\lambda^t$.
All this says is that $w\in \frank{\alpha}$ is the column growth word of some tableau of shape $\lambda/\mu$.
Observe that we mush have $\alpha\vDash |\lambda|-|\mu|$.
For instance, $w=\framebox{621}\hspace{1mm} \framebox{76432} \in \frank{\alpha}$ is $\lambda/\mu$-compatible for $\lambda=(7,6,4,2,2)$, $\mu=(5,5,2,1)$, and $\alpha=(3,5)$.
The reader can easily verify that $w$ is the column growth word of the tableau in Figure~\ref{fig:phi of w}.

Define
\begin{align}
\lrfrank{\lambda,\mu,\alpha}\coloneqq\{w\in\frank{\alpha}\suchthat \text{ $w$ is $\lambda/\mu$-compatible}\}.
\end{align}
What is special about this subset of frank words with column form $\alpha$?
There is an intimate link between LR tableaux and compatible words that we motivate by the following example.
Consider $\lambda=(7,6,4,3,2)$, $\mu=(6,4,4)$ and $\alpha=(3,4,1)$.
The reader can verify that $ \lrfrank{\lambda,\mu,\alpha}$ consists of
$\framebox{\textcolor{red}{321}}\hspace{1mm}\framebox{\textcolor{blue}{7621}}\hspace{1mm}\framebox{\textcolor{orange}{5}}$,
$\framebox{\textcolor{red}{621}}\hspace{1mm}\framebox{\textcolor{blue}{7321}}\hspace{1mm}\framebox{\textcolor{orange}{5}}$, and
$\framebox{\textcolor{red}{321}}\hspace{1mm}\framebox{\textcolor{blue}{6521}}\hspace{1mm}\framebox{\textcolor{orange}{7}}$.
Remarkably, these words are precisely column growth words of tableaux in Figure~\ref{fig:crystal action on LR tableaux}.
We make this connection precise.

Given $w\in \lrfrank{\lambda,\mu,\alpha}$, let $w^{(1)}\cdots w^{(m)}$ be its maximal column word factorization where $m\coloneqq\ell(\alpha)$.
Construct a Young tableau $\phi(w)$ of shape $\lambda/\mu$ and content $\reverse{\alpha}$ as follows: Let $\lambda^{(0)}\coloneqq\lambda$ and inductively define $\lambda^{(i)}$ for $1\leq i\leq m$ to be such that
$\lambda^{(i-1)}/\lambda^{(i)}$ is a horizontal strip with boxes in columns given by letters appearing in $w^{(i)}$.
Subsequently, fill the boxes of the horizontal strips $\lambda^{(i-1)}/\lambda^{(i)}$ with $m+1-i$  to obtain $\phi(w)$.
Note that $\lambda^{(m)}$ is $\mu$ and that $\phi(w)$ does indeed belong to $\YT(\lambda/\mu)$.
As an example, consider  $w=\framebox{\textcolor{red}{621}} \hspace{1mm}\framebox{\textcolor{blue}{76432}}$ which is $\lambda/\mu$-compatible for $\lambda=(7,6,4,2,2)$ and $\mu=(5,5,2,1)$.
The tableau $\phi(w)$ is shown in Figure~\ref{fig:phi of w}.
\begin{figure}[ht]
  \ytableausetup{mathmode,boxsize=1em}
  \begin{ytableau}
    *(red!40)2 & *(red!40)2\\
    *(gray!20) & *(blue!50)1\\
    *(gray!20) & *(gray!20) & *(blue!50)1 & *(blue!50)1\\
    *(gray!20) & *(gray!20) & *(gray!20) & *(gray!20) &*(gray!20) & *(red!40)2\\
    *(gray!20) & *(gray!20) & *(gray!20) & *(gray!20) &*(gray!20) & *(blue!50)1 & *(blue!50)1
  \end{ytableau}
  \caption{The tableau $\phi(w)$ corresponding to $\framebox{\textcolor{red}{621}} \hspace{1mm}\framebox{\textcolor{blue}{76432}}$.}
  \label{fig:phi of w}
\end{figure}

Define
\begin{align}
\lrfranktab{\lambda,\mu,\alpha}\coloneqq\{\phi(w)\suchthat w\in \lrfrank{\lambda,\mu,\alpha}\}.
\end{align}
Given a partition $\lambda$, denote the skew shape obtained by a $180^{\circ}$ rotation by $\rot{\lambda}$.
Our next lemma establishes that $\LRT(\lambda,\mu,\nu)=\lrfranktab{\lambda,\mu, \reverse{\nu}}$.
\begin{lemma}\label{lem:LR tableaux and LRfranktab}
  Let $\lambda,\mu$ and $\nu$ be partitions such that $\mu\subseteq \lambda$ and $|\nu|=|\lambda/\mu|$.
 We have
 \[
 	 \LRT(\lambda,\mu,\nu)=\lrfranktab{\lambda,\mu, \reverse{\nu}}.
 \]
\end{lemma}
\begin{proof}
Let $w\in \lrfrank{\lambda,\mu, \reverse{\nu}}$.
  As $\colform{w}=\reverse{\nu}$ and $w$ is frank, we know that $w$ is the column reading word of a tableau of skew shape $\rot{\nu^t}$.
 It follows that $\phi(w)\in \LRT({\lambda,\mu,\nu})$.
 This establishes $ \lrfranktab{\lambda,\mu, \reverse{\nu}}\subseteq \LRT(\lambda,\mu,\nu)$.
 The opposite inclusion relies follows since the column reading word of an LR tableaux is reverse-lattice. In particular, the column growth word of any $T\in \LRT(\lambda,\mu,\nu)$ is the column reading word of a Young tableaux of skew shape $\rot{\nu^t}$.
\end{proof}
Figure~\ref{fig:cgw of LR tableaux} depicts column growth words of tableaux in $\LRT(\lambda,\mu,\nu)$ (shown in Figure~\ref{fig:LR-examples}) as column reading words of Young tableaux of skew shape $\rot{\nu^t}$ where $\lambda=(7,6,4,3,2)$, $\mu=(6,4,4)$ and $\nu=(4,3,1)$.
\begin{figure}[ht]
$$
\ytableausetup{mathmode,boxsize=1em}
\begin{ytableau}
*(red!30)2& *(blue!50)3 & *(orange!75)7\\
*(gray!20) & *(blue!50)2 & *(orange!75)6\\
*(gray!20) & *(blue!50)1 & *(orange!75)5\\
*(gray!20) & *(gray!20) & *(orange!75)1
\end{ytableau}
\hspace{10mm}
\begin{ytableau}
*(red!30)2 & *(blue!50)6 & *(orange!75)7\\
*(gray!20) & *(blue!50)3 & *(orange!75)5\\
*(gray!20) & *(blue!50)1 & *(orange!75)2\\
*(gray!20) & *(gray!20) & *(orange!75)1
\end{ytableau}
\hspace{10mm}
\begin{ytableau}
*(red!30)3 & *(blue!50)6 & *(orange!75)7\\
*(gray!20) & *(blue!50)2 & *(orange!75)5\\
*(gray!20) & *(blue!50)1 & *(orange!75)2\\
*(gray!20) & *(gray!20) & *(orange!75)1
\end{ytableau}
$$
\caption{Column growth words of LR tableaux are frank words.}
\label{fig:cgw of LR tableaux}
\end{figure}

\subsection{Relating the two symmetric group actions}\label{subsec:sym_actions_tableaux_words}
Our next lemma connects the action of crystal reflection operators on $\lrfranktab{\lambda,\mu,\alpha}$  to the symmetric group action on $\lambda/\mu$-compatible frank words in $\lrfrank{\lambda,\mu,\alpha}$.
\begin{lemma}\label{lem:commuting action}
Consider $w\in \lrfrank{\lambda,\mu,\alpha}$ and $i$ satisfying $1\leq i\leq \ell(\alpha)-1$.
We have that $s_i(\phi(w))=\phi(s_{\ell(\alpha)-i}(w))$.
\end{lemma}
\begin{proof}
  Without loss of generality, we may assume that $\alpha$ has two parts.
  Suppose $\alpha=(p,q)\vDash n$
  Assume $w=w^{(1)}w^{(2)}$ where $|w^{(1)}|=p$ and $|w^{(2)}|=q$.
  We would like to establish that $s_1(\phi(w))=\phi(\iota(w))$.

  Assume $p<q$.
  Let $w^{(1)}=a_1\dots a_{p}$ and $w^{(2)}=b_1\dots b_{q}$.
  We have $a_1>\cdots > a_{p}$ and $b_1 > \cdots > b_{q}$.
  As $w$ is frank, it is the column reading word of a Young tableau of skew shape $\rot{(q,p)^{t}}$.
  Thus, $a_i\leq b_i$ for $1\leq i\leq p$.

Instead of computing $\iota(w)$ by way of rectifying the appropriate two-columned tableau, one may perform successive Schensted column insertions of the numbers $a_p$ down to $a_1$ starting from the single-columned tableau with column word $w^{(2)}$. See Figure~\ref{fig:iota via column insertion} for an example. Compare with Figure~\ref{fig:jdt on two columns} which established the same fact using jeu-de-taquin.
  \begin{figure}[h]
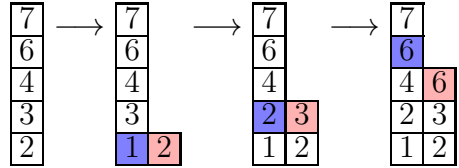

    \ytableausetup{mathmode,boxsize=1em}
    \begin{ytableau}
      7\\6\\4\\3\\2
    \end{ytableau}
    $\longrightarrow $
    \begin{ytableau}
      7\\6\\4\\3\\*(blue!50)1 & *(red!30)2
    \end{ytableau}
    $ \longrightarrow$
    \begin{ytableau}
      7\\6\\4\\*(blue!50)2 & *(red!30)3\\1 & 2
    \end{ytableau}
    $ \longrightarrow$
    \begin{ytableau}
      7\\*(blue!50)6\\4 & *(red!30)6\\2 & 3\\1 & 2
    \end{ytableau}
    \caption{Establishing that $\iota(\framebox{621} \hspace{1mm} \framebox{76432})=\framebox{76421}\hspace{1mm}\framebox{632}$ by column-inserting 1,2, and 6 into the single-columned tableau with column reading word $76432$. Blue boxes show the entries being inserted whereas  red boxes contain the entries  bumped.}
    \label{fig:iota via column insertion}
  \end{figure}

  During each intermediate step of this column-insertion procedure, the number $a_i$ being inserted into the current tableau bumps a distinct element from $\{b_1,\dots, b_q\}$.
  Furthermore this bumped entry is guaranteed to be strictly greater than the entries in the second column in the current tableau.
  Therefore, the insertion tableau is completely determined by the entries that get bumped.
  More precisely, for $i$ from $p$ down to $1$, define the integer ${\sf m}(i)$ recursively as follows.
  We define ${\sf m}(p)$ to be the largest integer $j$ such that $a_p\leq b_{j}$.
  Subsequently, for $i=p-1,\dots,1$, define ${\sf m}(i)$ to be the largest integer $j$ such that $j<{\sf m}(i+1) $ and $a_i\leq b_{j}$.
  Observe that in our Schensted column-insertion procedure, the entry $a_i$ bumps $b_{{\sf m}(i)}$.
  Therefore, the set of entries that get bumped is $\{b_{{\sf m}(i)}\suchthat 1\leq i\leq p\}$.
  For the example in Figure~\ref{fig:iota via column insertion}, we have ${\sf m}(3)=5$, ${\sf m}(2)=4$ and ${\sf m}(1)=2$. Therefore the set of entries that get bumped is $\{b_5,b_4,b_2\}=\{2,3,6\}$.

Consider $\crw{\phi(w)}=u_1\dots u_n$.
The word $v\coloneqq v_1\dots v_n$ obtained by recording the column to which each $u_i$ belongs gives us the weakly increasing arrangement of letters in $w$.
Furthermore, for $1\leq i\leq p$ (respectively $1\leq i\leq q$) the letter in $v$ corresponding to the $i$th $2$ (respectively 1) from the left in $\crw{\phi(w)}$ is equal to $a_i$ (respectively $b_i$).
Recall that the crystal reflection operator $s_1$ acting on $\phi(w)$ begins by pairing each $2$ in $\crw{\phi(w)}$ to the closest unpaired $1$ to its right.
Equivalently, in our current context, a $2$ corresponding to $a_i$ for some $1\leq i\leq p$ gets paired with the $1$ in  $\crw{\phi(w)}$ corresponding to  $b_{{\sf m}(i)}$.
We infer that the unpaired $1$s in $\crw{\phi(w)}$ correspond to those $b_j$ that are not bumped.
 These are precisely the $b_j$ that determine which $1$s in $\phi(w)$ turn into $2$s in computing $s_1(\phi(w))$.
Thus we infer that $s_1(\phi(w))=\phi(\iota(w))$.
This establishes the claim in the case $p<q$.
The case $p\geq q$ is similar and left to the reader.
\end{proof}
To illustrate the ideas in the preceding proof, Figure~\ref{fig:crystal action on phi of w} depicts the action of $s_1$ on the tableau $\phi(w)$ from Figure~\ref{fig:phi of w}, where $w= \framebox{621}\hspace{1mm}\framebox{76432}$.
From Figure~\ref{fig:iota via column insertion}, we see that the entries that do not get bumped  are $\{4,7\}$.
Also, note that $\crw{\phi(w)}=2211\textcolor{red}{1}21\textcolor{red}{1}$ where the unpaired $1$s are highlighted. In terms of the tableau $\phi(w)$, we see that the unpaired $1$s belong to columns $4$ and $7$. The tableau on the right in Figure~\ref{fig:crystal action on phi of w} is easily verified to be $\phi(\iota(w))$ as $\iota(w)=\framebox{76421}\hspace{1mm}\framebox{632}$.
\begin{figure}[h]
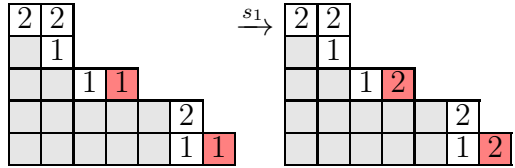

  \ytableausetup{mathmode,boxsize=1em}
  \begin{ytableau}
    2 & 2\\
    *(gray!20) & 1\\
    *(gray!20) & *(gray!20) & 1 & *(red!50)1\\
    *(gray!20) & *(gray!20) & *(gray!20) & *(gray!20) &*(gray!20) & 2\\
    *(gray!20) & *(gray!20) & *(gray!20) & *(gray!20) &*(gray!20) & 1 & *(red!50)1
  \end{ytableau}$ \xrightarrow[]{s_1}$
  \begin{ytableau}
     2 & 2\\
     *(gray!20) & 1\\
     *(gray!20) & *(gray!20) & 1 & *(red!50)2\\
     *(gray!20) & *(gray!20) & *(gray!20) & *(gray!20) &*(gray!20) & 2\\
     *(gray!20) & *(gray!20) & *(gray!20) & *(gray!20) &*(gray!20) & 1 & *(red!50)2
   \end{ytableau}
   \caption{Crystal reflection on $\phi(w)$ and its relation to the $\iota$ involution.}
   \label{fig:crystal action on phi of w}
 \end{figure}

We are ready to give a precise relation between  $\LRT^{\sigma}(\lambda,\mu,\nu)$ and $\lambda/\mu$-compatible frank words with a certain column form that generalizes Lemma~\ref{lem:LR tableaux and LRfranktab}.
\begin{proposition}\label{prop:sigma-LR and frank}
  Let $\lambda,\mu$ and $\nu$ be partitions such that $\mu\subseteq \lambda$ and $|\nu|=|\lambda/\mu|$.
  Let $w_0$ be the longest word in $\sgrp{\ell(\nu)}$.
  For $\sigma \in \sgrp{\ell(\nu)}$, we have
  \[
  	\LRT^{\sigma}(\lambda,\mu,\nu)=\lrfranktab{\lambda,\mu,(w_0\sigma w_0) \cdot \reverse{\nu}}.
  \]
\end{proposition}
\begin{proof}
 Lemma~\ref{lem:LR tableaux and LRfranktab} implies $\LRT(\lambda,\mu,\nu)=\lrfranktab{\lambda,\mu,\reverse{\nu}}$.
  Consider $T\in \LRT^{\sigma}(\lambda,\mu,\nu)$.
 We must have $T=\sigma(T')$ for a unique $T'\in \LRT(\lambda,\mu,\nu)$, which in turn implies that $T=\sigma(\phi(w'))$ for a unique $w'\in \lrfrank{\lambda,\mu,\reverse{\nu}}$.
  Lemma~\ref{lem:commuting action} implies that $T=\phi(w_0\sigma w_0(w'))$.
  As $\colform{w_0\sigma w_0(w')}=(w_0\sigma w_0) \cdot\reverse{\nu}$, we conclude that
  \begin{align}
  \LRT^{\sigma}(\lambda,\mu,\nu)\subseteq\lrfranktab{\lambda,\mu,(w_0\sigma w_0) \cdot \nu^r}.
  \end{align}
  A simple cardinality count implies this inclusion must be an equality.
\end{proof}
Proposition~\ref{prop:sigma-LR and frank} states that elements of $\LRT^{\sigma}(\lambda,\mu,\nu)$ are in bijection with $\lambda/\mu$-compatible frank words with column form $w_0\sigma w_0\cdot \reverse{\nu}$.
As an example, consider $\LRT^{\sigma}(\lambda,\,\mu,\,\nu)$ from Figure~\ref{fig:crystal action on LR tableaux} where $\lambda=(7,6,4,3,2)$, $\mu=(6,4,4)$, $\nu=(4,3,1)$ and $\sigma=s_1s_2$.
Note that the column growth words of these tableaux are indeed frank words with column form $(w_0 s_1s_2 w_0)\cdot (1,3,4)=(s_2s_1)\cdot(1,3,4)=(3,4,1)$. Figure~\ref{fig:cgw of sigma-LR tableaux} shows these frank words as column reading words of tableaux with column lengths $3$, $4$ and $1$ read from left to right.
We encourage the reader to obtain these tableaux by performing jeu-de-taquin slides to  tableaux in Figure~\ref{fig:cgw of LR tableaux}.
\begin{figure}[ht]
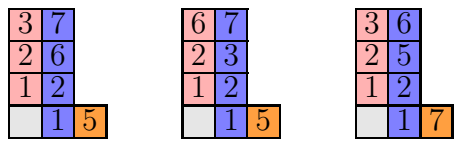

$$
\ytableausetup{mathmode,boxsize=1em}
\begin{ytableau}
*(red!30)3& *(blue!50)7 \\
*(red!30)2 & *(blue!50)6 \\
*(red!30)1 & *(blue!50)2 \\
*(gray!20) & *(blue!50)1 &  *(orange!75)5
\end{ytableau}
\hspace{10mm}
\begin{ytableau}
*(red!30)6 & *(blue!50)7 \\
*(red!30)2 & *(blue!50)3 \\
*(red!30)1 & *(blue!50)2 \\
*(gray!20) & *(blue!50)1 &  *(orange!75)5
\end{ytableau}
\hspace{10mm}
\begin{ytableau}
*(red!30)3 & *(blue!50)6 \\
*(red!30)2 & *(blue!50)5 \\
*(red!30)1 & *(blue!50)2 \\
*(gray!20) & *(blue!50)1 &  *(orange!75)7
\end{ytableau}
$$
\caption{Column growth words of tableaux in Figure~\ref{fig:crystal action on LR tableaux} are frank words.}
\label{fig:cgw of sigma-LR tableaux}
\end{figure}

\section{Noncommutative LR coefficients and frank words}\label{sec:back to CT}
Now that we understand tableaux in $\LRT^{\sigma}(\lambda,\mu,\nu)$ as certain $\lambda/\mu$-compatible words with a prescribed column form, we are ready to establish the connection to noncommutative LR coefficients.
We need some preliminary lemmas.
\begin{lemma}\label{lem: rect stan equaling Q}
  For any Young tableau $T$, we have $\rect{\stan{T}}=\evac(\qtab(\cgw{T}))^{t}$.
\end{lemma}
\begin{proof}
  We sketch the proof and follow the exposition in \cite{Lothaire}.
  Let $T'=\stan{T}$ and suppose that $T$ has $n$ boxes.
  Consider the biword $\left[\begin{array}{c}u \\ v\end{array}\right]$ where $u\coloneqq u_1\cdots u_n$ is the longest word in $\mathfrak{S}_n$ and $v\coloneqq v_1\cdots v_n$ is obtained by recording the column in $T'$ to which $u_i$ belongs.
In other words, $v=\cgw{T'}=\cgw{T}$.
 Consider  biwords $\left[\begin{array}{c}u' \\ v'\end{array}\right]$ and $\left[\begin{array}{c}u'' \\ v''\end{array}\right]$, where $u'$ (respectively $v''$) in the weakly increasing rearrangement of $u$ (respectively $v$) and $v'$ (respectively $u''$) is the rearrangement induced by the aforementioned sorting.
 Then we have that
\[
v'= \reverse{\cgw{T'}} \text{ and }
u''= \crw{T'}.
\]
Note that $\rect{\stan{T}}=\rect{T'}=\ptab(\crw{T'})=\ptab(u'')$. By \cite[Proposition 5.3.9]{Lothaire}, this equals $\qtab(v')$ and by \cite[Corollary A1.2.11]{stanley-ec2} which relates reversal to evacuation, the claim now follows.
\end{proof}
 For the leftmost tableau $T$ in Figure~\ref{fig:crystal action on LR tableaux}, its standardization rectifies to the tableau in the middle in Figure~\ref{fig:rect-stan via evacuation}.
We have $\cgw{T}=\framebox{321}\hspace{1mm} \framebox{7621} \hspace{1mm} \framebox{5}$, and $\qtab(\cgw{T})$ is shown on the right.
We invite the reader to verify that $\evac(\qtab(\cgw{T}))$ is indeed the tableau in the middle upon transposing.
\begin{figure}[h]
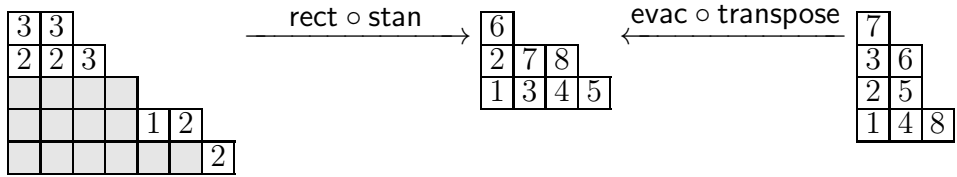

$$
\ytableausetup{mathmode,boxsize=1em}
\begin{ytableau}
3 & 3\\
2 &2 &3\\
*(gray!20) &*(gray!20) & *(gray!20) &*(gray!20)\\
*(gray!20) &*(gray!20) &*(gray!20) &*(gray!20) & 1 & 2\\
*(gray!20) &*(gray!20) &*(gray!20) &*(gray!20) &*(gray!20) &*(gray!20) & 2\\
\end{ytableau}
\xrightarrow{\makebox[2.8cm]{$\mathsf{rect}\circ \mathsf{stan}$}}
\begin{ytableau}
6 \\
2 & 7 & 8\\
1 & 3 & 4 & 5
\end{ytableau}
\xleftarrow{\makebox[2.8cm]{$\mathsf{evac}\circ \mathsf{transpose}$}}
\begin{ytableau}
7\\
3 & 6\\
2 & 5\\
1 & 4 &8
\end{ytableau}
$$
\caption{A demonstration of the claim $\rect{\stan{T}}=\evac(\qtab(\cgw{T}))^{t}$.}
\label{fig:rect-stan via evacuation}
\end{figure}

\begin{lemma}\label{lem:unique tableau  with shape=descent composition}
Let  $\alpha$ be a composition.
There is a unique Young tableau $T$ of shape $\sort(\alpha)$ and descent composition $\alpha$.
Furthermore, this tableau satisfies $\rho^{-1}(T)=\tau_{\alpha}$.
\end{lemma}
\begin{proof}
Observe that the first column of $T$ when read from bottom to top is forced to be $1, 1+\alpha_1,1+\alpha_1+\alpha_2,\dots, 1+\sum_{j=1}^{\ell(\alpha)-1}\alpha_j$.
The first claim follows easily from this observation, and our second claim follows from applying the map $\rho$ to the Young tableau $T$ obtained earlier.
We leave the details to the interested reader.
\end{proof}
We finally arrive at the key proposition that is utilized in the proof of our central result which  is Theorem~\ref{thm:main result 1}.
\begin{proposition}\label{prop:crystal action permutes parts}
If $T\in \LRT^{\sigma}(\lambda,\mu,\nu)$, then $\rect{\stan{T}}=\rho(\tau_{\alpha})$ where $\alpha=\sigma\cdot \nu$.
\end{proposition}
\begin{proof}
Throughout, assume that $\sigma\cdot \nu=\alpha$ and set $n\coloneqq |\alpha|$.
By Lemma~\ref{lem: rect stan equaling Q} we have
\begin{align}\label{eqn:final_1}
\rect{\stan{T}}=\evac(\qtab(\cgw{T}))^{t}.
\end{align}
Since $T\in\LRT^{\sigma}(\lambda,\mu,\nu)$,  Proposition~\ref{prop:sigma-LR and frank} implies
\begin{align}
\cgw{T}\in \lrfrank{\lambda,\mu,\reverse{\alpha}}.
\end{align}
Therefore, the shape underlying $\qtab(\cgw{T})$  is $\sort(\alpha)^t=\nu^{t}$.
Using the preceding fact along with  \eqref{eqn:final_1}, we infer that $ \rect{\stan{T}}$ has shape $\nu$.

Note that the descent set of $\cgw{T}$ is $ \{n-i\suchthat i\in [n-1]\setminus\set(\alpha)\}$, which therefore is also the descent set of $\qtab(\cgw{T})$.
It follows that the descent set of $\evac(\qtab(\cgw{T}))$ is $ [n-1]\setminus\set(\alpha)$, which in turn implies that the descent set of $\rect{\stan{T}}=\evac(\qtab(\cgw{T}))^{t}$ is $\set(\alpha)$.
Thus we have established that $\rect{\stan{T}}$ has shape $\sort(\alpha)$ and descent composition $\alpha$.
Lemma~\ref{lem:unique tableau with shape=descent composition} proves the proposition.
\end{proof}

This also completes the proof of our main theorem.
A remarkable aspect of Proposition~\ref{prop:crystal action permutes parts} is that the symmetric group action on LR tableaux via crystal reflection operators translates to the usual permutation action on the parts of the shape underlying the rectification, after applying the $\rho$ map.

We briefly describe an equivalent interpretation that involves box-adding operators on compositions.
By Proposition~\ref{prop:sigma-LR and frank}, we know that elements of $\LRT^{\sigma}(\lambda,\mu,\nu)$ may be constructed by computing $\lambda/\mu$-compatible words $w$ satisfying $\colform{w}=(w_0\sigma w_0)\cdot \reverse{\nu}$, where $w_0$ is the longest word in $\sgrp{\ell(\nu)}$.
Equivalently, these words are exactly the column growth words of standardizations of  tableaux in $\LRT^{\sigma}(\lambda,\mu,\nu)$.
Therefore, they may be identified with certain saturated chains in Young's lattice from $\mu$ to $\lambda$.
By applying the map $\rho_{\beta}$ where $\sort(\beta)=\mu$, these chains may be interpreted as chains in $\Lc$ from $\beta$ to certain compositions $\gamma$ that satisfy $\sort(\gamma)=\lambda$.
Fixing a $\gamma$ and counting these chains allows us to compute $C_{\alpha\beta}^{\gamma}$ where $\alpha=\sigma\cdot \nu$.

We make the preceding discussion precise by phrasing the result in the language of box-adding operators on compositions introduced in \cite{Tewari-MN} following the seminal work of Fomin \cite{Fomin-Greene}.
Given $i\geq 1$ and a composition $\alpha$ that has at least one part equaling $i-1$, we define $\addt_i(\alpha)$ to be the unique composition  covering $\alpha$ in $\Lc$ where the new box occurs in the $i$-th column.
Given a word $w=w_1\dots w_n$, define  $\addt_{w}\coloneqq
\addt_{w_1}\cdots \addt_{w_n}$.
We have the following corollary.
\begin{corollary}\label{cor:box-adding and LR coefficients}
Let $\alpha,\,\beta$ and $\gamma$ be compositions such that $|\beta|+|\alpha|=|\gamma|$. Let $(\lambda,\mu,\nu)$ be the triple $(\sort(\gamma), \sort(\beta),\sort(\alpha))$.
If $\sigma\in \mathfrak{S}_{\ell(\nu)}$ satisfies $\alpha=\sigma\cdot \nu$, then
\[
C_{\alpha\beta}^{\gamma}=|\{w\in \lrfrank{\lambda,\mu,(w_0\sigma w_0)\cdot\reverse{\nu}} \suchthat \addt_w(\beta)=\gamma\}|.
\]
\end{corollary}
The cases in  Corollary~\ref{cor:box-adding and LR coefficients} where  $\sigma$ is either the identity permutation or the longest permutation in $\sgrp{\ell(\nu)}$ correspond to left and right LR rules of \cite{BTvW}.
We remark here that the two LR rules in \cite{BTvW} were proved by different approaches. Thus,  not only does our Corollary~\ref{cor:box-adding and LR coefficients} generalize these LR rules, it provides a uniform proof that works in all cases.

\section*{Acknowledgements}
We are very grateful to Sami Assaf for suggesting that we study our noncommutative LR coefficients from a crystal perspective and for numerous helpful discussions. We are also thankful to  Sara Billey and Steph van Willigenburg for various discussions.

\end{document}